\documentclass[12pt]{amsart}
\usepackage[margin=1.1in]{geometry}
\usepackage{amsmath,amssymb,amsthm,mathtools}
\usepackage[colorlinks=true,linkcolor=blue,citecolor=blue]{hyperref}

\newtheorem{theorem}{Theorem}[section]
\newtheorem{proposition}[theorem]{Proposition}
\newtheorem{lemma}[theorem]{Lemma}

\theoremstyle{definition}
\newtheorem{remark}[theorem]{Remark}

\newcommand{\R}{\mathbb{R}}

\newcommand{\Lou}{\mathcal{L}}
\newcommand{\Tg}{T_\gamma}
\newcommand{\Per}{\operatorname{Per}}

\title[Failure of a BM-type inequality for the Gaussian torsional rigidity]{Failure of a Brunn--Minkowski-type inequality\\
for the Gaussian torsional rigidity}
\author{Xuan Hien Nguyen}
\address[Xuan Hien Nguyen]{Iowa State University}
\email{\href{mailto: xhnguyen@iastate.edu}{\nolinkurl{xhnguyen@iastate.edu}}}

\author{Alina Stancu}
\address[Alina Stancu]{Concordia University}
\email{\href{mailto: alina.stancu@concordia.ca}{\nolinkurl{alina.stancu@concordia.ca}}}

\date{\today}

\begin{document}
\maketitle
\begin{abstract}
Let $u$ be the torsion function for the Ornstein--Uhlenbeck operator on a bounded domain $\Omega \subset \R^n$, i.e.,
the solution of $\Delta u - x \cdot \nabla u=-1$ in $\Omega$ with $u=0$ on $\partial\Omega$. Let $\Tg(\Omega)=\int_\Omega u\,d\gamma$ be the Gaussian torsional
rigidity.  We prove that the Brunn--Minkowski-type inequality
$\Tg((1-t)\Omega_0+t\Omega_1)^{\alpha}\le
(1-t)\Tg(\Omega_0)^{\alpha}+t\Tg(\Omega_1)^{\alpha}$ fails for every
exponent $\alpha>0$ and in every dimension $n\ge2$, for a pair of convex
bodies centrally symmetric with respect to the origin, which may be taken
smooth with positive curvature.
This answers Conjecture~1.4  for any $n \geq 2$, and Question~(Q), of Mar\'in Sola and Salerno
in the negative.  The mechanism is a first-order lower bound for $\Tg$ at
a ball $\Omega_0$ under Minkowski perturbations. When the perturbing body $\Omega_1$ has small
torsion and large mean width, $\Tg((1-t) \Omega_0 + t \Omega_1)$ increases to first order. Since the Minkowski combination has
larger torsion than both endpoints, no exponent can repair the
inequality. For $n=1$, convexity with the optimal exponent $\frac13$
holds on symmetric intervals by results of the same authors, \cite{MSS26}. We
prove that the logarithm of the torsion is neither convex nor concave
along Minkowski combinations of symmetric intervals, that no non-zero
exponent yields concavity, and that convexity fails for every positive
exponent when one set is a union of two intervals or when the sets are
reflected off-center intervals.
\end{abstract}



\section{Introduction}\label{sec:intro}

Let $\gamma$ be the standard Gaussian measure on $\R^n$,
$d\gamma=(2\pi)^{-n/2}e^{-|x|^2/2}\,dx$, and let
\[
\Lou u=\Delta u-x\cdot\nabla u
\]
be the Ornstein--Uhlenbeck operator, so that
$\int\nabla u\cdot\nabla v\,d\gamma=\int(-\Lou u)\,v\,d\gamma$ for smooth
$u$ and compactly supported $v$.  For a bounded open set
$\Omega\subset\R^n$, the torsion function of $\Omega$ is the solution
$u=u_\Omega$ of
\begin{equation}\label{eq:torsion}
\Lou u=-1\ \text{ in }\Omega,\qquad u=0\ \text{ on }\partial\Omega,
\end{equation}
and the Gaussian torsional rigidity is $\Tg(\Omega)=\int_\Omega u\,d\gamma$.
For a convex body $K$, that is, a compact convex set with nonempty
interior, we write $\Tg(K):=\Tg(\operatorname{int}K)$. Throughout,
centrally symmetric means centrally symmetric with respect to the
origin, that is, $\Omega=-\Omega$.

Mar\'in Sola and Salerno \cite{MSS26} studied Brunn--Minkowski-type
inequalities for $\Tg$. They proved that
$t\mapsto\Tg((1-t)B_{R_0}+tB_{R_1})^{1/3}$ is convex for balls in $\R^n$, and conjectured (Conjecture~1.4 of \cite{MSS26}) that
\begin{equation}\label{eq:BMconj}
\Tg\bigl((1-t)\Omega_0+t\Omega_1\bigr)^{1/3}\le
(1-t)\Tg(\Omega_0)^{1/3}+t\Tg(\Omega_1)^{1/3}
\end{equation}
for centrally symmetric bounded convex $\Omega_0,\Omega_1\subset\R^n$ and
$t\in[0,1]$, where
\[
(1-t)\Omega_0+t\Omega_1
=\bigl\{(1-t)x_0+tx_1:\ x_0\in\Omega_0,\ x_1\in\Omega_1\bigr\}
\]
is the Minkowski combination of $\Omega_0$ and $\Omega_1$. They also observed numerically that an analogous statement with exponent $1/(n+2)$ fails when $n \in \{3, 5, 10\}$ as the corresponding inequality changes sign on balls, stated that this phenomenon seems to occur for any $n \geq 3$, and asked if this exponent is, however, sharp for $n \in \{1, 2\}$ (Question~(Q) in \cite{MSS26}).  Related work on
log-concavity for the Ornstein--Uhlenbeck operator includes
\cite{CFLS24,CQSell,CQS26,Qin25}.  In the classical setting of the Laplace operator, inequalities of
this type rest on concavity properties of the torsion function, such as
convexity of the superlevel sets \cite{ML71} or the concavity of $\sqrt u$ \cite{Ken85}.  For
$\Lou$ no such property is available. The companion work \cite{NgConc} addresses the concavity of the torsion function with diverging conclusions even for smooth, uniformly convex domains.

We prove the following.

\begin{theorem}\label{thm:main-BM}
Let $n\ge2$. There exist $L>0$, $\varepsilon>0$, and $t_1>0$ such that
$\Omega_0=B_1\subset\R^n$ and
$\Omega_1=([-L,L]\times\{0\}^{n-1})+B_\varepsilon$ satisfy
\[
\Tg\bigl((1-t)\Omega_0+t\Omega_1\bigr)>
\max\bigl(\Tg(\Omega_0),\Tg(\Omega_1)\bigr)
\qquad\text{for all }t\in(0,t_1).
\]
Consequently, for every $\alpha>0$, the function
$t\mapsto\Tg((1-t)\Omega_0+t\Omega_1)^{\alpha}$ is not convex on $[0,1]$.
In particular, both the conjectured inequality \eqref{eq:BMconj}
 and  the inequality of Question~(Q) of \cite{MSS26}
with exponent $1/(n+2)$ fail for any dimension $n \geq 2$. For $n=2$ and $n=3$, one may take $L=2$
and $\varepsilon=\frac1{10}$. The same holds with $\Omega_1$ replaced by
a centrally symmetric convex body with $C^\infty$ boundary and positive
curvature.
\end{theorem}

The mechanism is a first-order lower bound: at a ball, under Minkowski
perturbations, $\Tg$ grows at least at a rate proportional to the excess
of the perturbing body's perimeter over the ball's
(Proposition~\ref{prop:variation}), so a body of small torsion and large
perimeter increases $\Tg$ to first order.

The construction is not specific to the plane. In $\R^n$, the corresponding lower bound
is proportional to the mean width of the perturbing body, which is again
linear under Minkowski addition, and a needle
$([-L,L]\times\{0\}^{n-1})+B_\varepsilon$ of large mean width and
small torsion violates every inequality of the form \eqref{eq:BMconj}.
Theorem~\ref{thm:main-n} proves Theorem~\ref{thm:main-BM} for $n\ge3$,
with explicit $L$ and $\varepsilon$. In particular, the inequality of
Question~(Q) is negative in every dimension $n \geq 2$.

For $n=1$, the picture is different. A bounded open connected centrally
symmetric subset of $\R$ is a symmetric interval $B_R=(-R,R)$, and for
these sets Mar\'in Sola and Salerno prove that
$t\mapsto\Tg((1-t)B_{R_0}+tB_{R_1})^{1/3}$ is convex and that the
exponent $\frac13$ is optimal \cite{MSS26}. We prove the following.

\begin{theorem}\label{thm:main-1}
Let $n=1$, and for $R_0,R_1>0$ let $R_t=(1-t)R_0+tR_1$.
\begin{enumerate}
\item[(i)] The function $t\mapsto\log\Tg(B_{R_t})$ is neither convex
for all pairs $R_0,R_1$ nor concave for all pairs, and for every
$\alpha\in\R\setminus\{0\}$ there exist $R_0\neq R_1$ for which
$t\mapsto\Tg(B_{R_t})^{\alpha}$ is not concave.
\item[(ii)] For every $\alpha>0$, the function
$t\mapsto\Tg\bigl((1-t)\Omega_0+t\Omega_1\bigr)^{\alpha}$ is not convex
on $[0,1]$, and neither is
$t\mapsto\log\Tg\bigl((1-t)\Omega_0+t\Omega_1\bigr)$, both when
$\Omega_0=(-\delta,\delta)$ and
$\Omega_1=(-2-\delta,-2+\delta)\cup(2-\delta,2+\delta)$ with
$0<\delta\le\frac1{10}$, and when $\Omega_0=(2-\delta,2+\delta)$ and
$\Omega_1=-\Omega_0$.
\end{enumerate}
\end{theorem}

Theorem~\ref{thm:main-1} is proved in Section~\ref{sec:oneD}: part (i)
is Theorem~\ref{thm:oneD} (iii)--(iv), and it makes the numerical
observations of Example~3.2 in \cite{MSS26} rigorous, part (ii) is
Propositions~\ref{prop:disconnected} and~\ref{prop:noncentered}. Thus,
the convexity of Question~(Q) for $n=1$ holds on symmetric intervals
with the optimal exponent $\frac13$, and fails for every positive
exponent in the stated generality of open, bounded, centrally symmetric
sets. Together with Theorem~\ref{thm:main-BM}, this settles
Question~(Q).

This note is structured as follows. Section~\ref{sec:prelim} collects notation and standard facts,
Section~\ref{sec:BM} proves Theorem~\ref{thm:main-BM} for dimension 2,
Section~\ref{sec:higherdim} extends it to $\R^n$, and
Section~\ref{sec:oneD} settles the one-dimensional case.

\section{Preliminaries}\label{sec:prelim}

\subsection{Regularity and comparison}\label{sec:regularity}

Let $\Omega\subset\R^n$ be a bounded convex domain. Every boundary point of
$\Omega$ admits an exterior half-space, hence an exterior ball, so
\eqref{eq:torsion} has a unique solution
$u\in C^{2,\alpha}_{\mathrm{loc}}(\Omega)\cap C(\overline\Omega)$
\cite[Thm.~6.13]{GT}, and $u\in C^\infty(\Omega)$ since the coefficients
of $\Lou$ are smooth \cite[Thm.~6.17]{GT}. The
weak maximum principle for $\Lou$ holds since $\Lou$ has no zeroth-order
term \cite[Thm.~3.3]{GT}, and the strong maximum principle gives $u>0$ in
$\Omega$.

The operator is symmetric with respect to $\gamma$:
$e^{-|x|^2/2}\Lou u=\operatorname{div}\bigl(e^{-|x|^2/2}\nabla u\bigr)$.
Hence a function $\varphi\in W^{1,2}_{\mathrm{loc}}$ solves $\Lou\varphi=0$
weakly if and only if it solves the divergence-form equation
$\operatorname{div}(e^{-|x|^2/2}\nabla\varphi)=0$ weakly.

\subsection{Variational characterization}\label{sec:variational}

Let $H^1_0(\Omega,\gamma)$ be the closure of $C_c^\infty(\Omega)$ in the
norm $\|v\|^2=\int_{\Omega}(v^2+|\nabla v|^2)\,d\gamma$. For bounded $\Omega$, this
is $H^1_0(\Omega)$ with an equivalent norm.

\begin{lemma}\label{lem:variational}
Let $\Omega\subset\R^n$ be a bounded open set with torsion function $u$.
Then
\[
\Tg(\Omega)=\max\Bigl\{\int_\Omega\bigl(2v-|\nabla v|^2\bigr)\,d\gamma:\
v\in H^1_0(\Omega,\gamma)\Bigr\},
\]
and the maximum is attained only at $v=u$.
\end{lemma}

\begin{proof}
Write $v=u+\varphi$ with $\varphi\in H^1_0(\Omega,\gamma)$. By the weak
form of \eqref{eq:torsion},
$\int_{\Omega}\nabla u\cdot\nabla\varphi\,d\gamma=\int\varphi\,d\gamma$, so
\[
\int_{\Omega}(2v-|\nabla v|^2)\,d\gamma
=\int_{\Omega}(2u-|\nabla u|^2)\,d\gamma-\int_{\Omega}|\nabla\varphi|^2\,d\gamma
=\Tg(\Omega)-\int_{\Omega}|\nabla\varphi|^2\,d\gamma \leq \Tg(\Omega),
\]
where the last equality uses $\int|\nabla u|^2d\gamma=\int u\,d\gamma$,
again by the weak form with test function $u$. Equality in the inequality holds only if
$\nabla\varphi=0$, hence $\varphi=0$ by the Friedrichs inequality on the
bounded set $\Omega$.
\end{proof}

\subsection{The ball}\label{sec:ball}

For $n\ge1$ let $u_n$ be the radial solution of
\begin{equation}\label{eq:radialODE}
u_n''+\frac{n-1}{r}u_n'-ru_n'=-1,\qquad u_n(0)=0,\quad u_n'(0)=0 .
\end{equation}
The torsion function of $B_R\subset\R^n$ is $u_n(|x|)-u_n(R)$. For $n=2$,
the first-order equation for $u_2'$ integrates explicitly:
\begin{equation}\label{eq:u2}
u_2'(r)=-\frac{e^{r^2/2}-1}{r},\qquad
u_2''(r)=-\frac{e^{r^2/2}(r^2-1)+1}{r^2}.
\end{equation}

\subsection{Support functions and Minkowski addition}\label{sec:convex}

For a convex body $K\subset\R^n$, let $h_K(\xi)=\max_{z\in K}z\cdot\xi$ be
its support function as a function on $\mathbb{S}^{n-1}$. Minkowski addition of convex bodies corresponds to addition of support functions,
$h_{K+L}=h_K+h_L$, and $h_{\lambda K}=\lambda h_K$ for $\lambda\ge0$
\cite[Thm.~1.7.5]{Sch14}. The
Hausdorff distance satisfies
$d_H(K,L)=\|h_K-h_L\|_{L^\infty(\mathbb{S}^{n-1})}$, so
$d_H(K+M,L+M)=d_H(K,L)$. In the plane, with $h(\theta)=h_K(\cos\theta,
\sin\theta)$, Cauchy's formula gives
$\Per(K)=\int_0^{2\pi}h(\theta)\,d\theta$, and $|h'|\le\max_{ z\in K}|z|$ almost
everywhere. A convex body $K$ is of class $C^\infty_+$ if $\partial K$ is a
$C^\infty$ hypersurface with positive Gauss curvature. Then $h_K$ is
$C^\infty$ on $\R^n\setminus\{0\}$ when extended by homogeneity $h_K(x)=|x|h_K(x/|x|)$,  and the principal radii of curvature of
$\partial K$ at the point with outer normal $\xi\in \mathbb{S}^{n-1}$ are the
eigenvalues of the Hessian of $h_K$ restricted to $\xi^\perp$
\cite[\S2.5]{Sch14}. Every convex body is a Hausdorff limit of bodies of
class $C^\infty_+$, and if $K$ is centrally symmetric the approximants can
be chosen centrally symmetric \cite[\S3.4]{Sch14}, \cite[p.~438]{Sch84}.

\section{The planar case}\label{sec:BM}

\subsection{The first variation at a ball}\label{sec:variation}

\begin{lemma}\label{lem:ballclosed}
For $B_R\subset\R^2$ with $R>0$ and $Z=R^2/2$,
\[
\Tg(B_R)=\int_0^{Z}\frac{\cosh z-1}{z}\,dz
=\sum_{m\ge1}\frac{Z^{2m}}{2m\,(2m)!} .
\]
\end{lemma}

\begin{proof}
By Lemma~\ref{lem:variational} and its proof,
$\Tg(B_R)=\int_{B_R}|\nabla u|^2\,d\gamma$ with $u=u_2-u_2(R)$, so
$\Tg(B_R)=\int_0^Ru_2'(r)^2e^{-r^2/2}r\,dr$. Substituting $z=r^2/2$ and
using \eqref{eq:u2},
$u_2'(r)^2e^{-r^2/2}r\,dr=(e^z-1)^2e^{-z}\,dz/(2z)$, and
$(e^z-1)^2e^{-z}=2(\cosh z-1)$. Expanding
$(\cosh z-1)/z=\sum_{m\ge1}z^{2m-1}/(2m)!$ and integrating gives the
series.
\end{proof}

\begin{lemma}\label{lem:radial}
Let $\Omega_1\subset B_{R_1}\subset\R^2$ be a convex body containing the origin, let $R>0$, and
for $0<t\le t_0=\min\{\frac12,\frac{R}{8R_1}\}$ let
$\Omega_t=(1-t)B_R+t\Omega_1$, with support function $h_t=(1-t)R+th_1$ and
radial function $\rho_t$. Then for all $\theta$,
\begin{equation}\label{eq:radial-i}
h_t(\theta)-\frac{\pi R_1^2}{R}\,t^2\le\rho_t(\theta)\le h_t(\theta),
\end{equation}
and for almost every $\theta$,
\begin{equation}\label{eq:radial-ii}
|\rho_t'(\theta)|\le3R_1t .
\end{equation}
\end{lemma}

\begin{proof}
Write $e_\theta=(\cos\theta,\sin\theta)$. The upper bound in
\eqref{eq:radial-i} is $\rho_t(\theta)=\rho_t(\theta)e_\theta\cdot e_\theta
\le h_t(\theta)$ by the definition of the support function. Let $x=\rho_t(\theta)e_\theta\in\partial\Omega_t$ and
let $\nu$ be an outer unit normal to $\Omega_t$ at $x$, at angle
$\varphi^*$. Every boundary point of a Minkowski sum decomposes as a sum of
boundary points with a common normal: in any
decomposition $x=a+b$ with $a\in(1-t)B_R$ and $b\in t\Omega_1$, the two
terms of $a\cdot\nu+b\cdot\nu=x\cdot\nu=(1-t)h_{B_R}(\nu)+t\,h_1(\nu)$
are bounded by the corresponding support values, so both are equalities.
Hence $x=(1-t)R\nu+t\zeta$ with
$\zeta\in\partial\Omega_1$, so that $x\times\nu=t(\zeta\times\nu)$ and the
angle $\beta=|\theta-\varphi^*|$ between $e_\theta$ and $\nu$ satisfies
\begin{equation}\label{eq:beta}
\sin\beta=\frac{|x\times\nu|}{|x|}\le\frac{tR_1}{(1-t)R}
\le\frac{2R_1t}{R}\le\frac14,
\end{equation}
because $|x\times\nu| =t|\zeta\times\nu| \leq tR_1$ since $\Omega_1\subset B_{R_1}$, and in the denominator, we use the assumption $0\in\Omega_1$. Indeed, since $\zeta$ is a supporting point of $\Omega_1$ with outward normal $\nu$, $ \zeta\cdot\nu=h_{1}(\nu)\geq0$, and
$|x|
 \geq x\cdot\nu
 =(1-t)R+t\,\zeta\cdot\nu
 \geq(1-t)R.$
In particular, $\cos\beta=x\cdot\nu/|x|>0$, so $\beta<\frac\pi2$.

Since $x\cdot\nu=h_t(\varphi^*)$,
$\rho_t(\theta)=h_t(\varphi^*)/\cos\beta\ge h_t(\varphi^*)\ge
h_t(\theta)-tR_1\beta$, using $|h_t'|=t|h_1'|\le tR_1$ almost everywhere.
On $[0, \frac{\pi}{2}]$,
$\beta\le\frac\pi2\sin\beta\le\pi R_1t/R$ which gives
\eqref{eq:radial-i}. 

At points of differentiability of $\rho_t$, $
 x'(\theta)
 =\rho_t'(\theta)e_\theta
 +\rho_t(\theta)e_\theta^\perp$  is tangential, so the polar
identity $\rho_t'=-\rho_t\,(\nu\cdot e_\theta^\perp)/(\nu\cdot e_\theta)$
gives $|\rho_t'|=\rho_t\tan\beta$. Now
$\rho_t\le R+tR_1\le\frac98R$ and, by \eqref{eq:beta}, $\cos\beta\ge0.96$, so
$|\rho_t'|\le\frac98R\cdot\frac{2R_1t}{0.96R}\le3R_1t$.
\end{proof}

\begin{proposition}\label{prop:variation}
Let $\Omega_1\subset B_{R_1}\subset\R^2$ be a convex body containing the origin, $R>0$, and
\[
M=\frac1{2\pi}\int_0^{2\pi}\bigl(h_1(\theta)-R\bigr)\,d\theta
=\frac{\Per(\Omega_1)-\Per(B_R)}{2\pi},
\ \ 
c_R=R\,u_2'(R)^2e^{-R^2/2}.
\]
Then there exist $t_1,C>0$ depending only on $R,R_1$ such that
\[
\Tg\bigl((1-t)B_R+t\Omega_1\bigr)\ge\Tg(B_R)+c_RMt-Ct^2
\qquad\text{for }0<t\le t_1 .
\]
In particular, if $\Per(\Omega_1)>2\pi R$, then
$\Tg((1-t)B_R+t\Omega_1)>\Tg(B_R)$ for all sufficiently small $t>0$.
\end{proposition}

\begin{proof}
It suffices to prove the claim for some $t_1 \leq t_0$, with $t_0$ as in
Lemma~\ref{lem:radial}. Consider $u=u_2-u_2(R)$ and $\rho=\rho_t$ the radial function of
$\Omega_t=(1-t)B_R+t\Omega_1$. Define
\[
v_t(re_\theta)=u\Bigl(\frac{Rr}{\rho(\theta)}\Bigr),\qquad
0\le r<\rho(\theta).
\]
Since
$r^{-1}\partial_\theta v_t=-u'\,R\rho'/\rho^2$,
$|\nabla v_t|^2=|\partial_r v_t|^2
+\frac{1}{r^2}|\partial_\theta v_t|^2=\frac{R^2}{\rho^2}\bigl(1+\frac{\rho'^2}{\rho^2}\bigr)
u'\bigl(\frac{Rr}{\rho}\bigr)^2$ almost everywhere. Since
$\rho_t\ge(1-t)R\ge R/2$ and $\rho_t\in W^{1,\infty}(\mathbb{S}^1)$, this formula
shows that $v_t\in W^{1,\infty}(\Omega_t)$. Its trace on
$\partial\Omega_t$ is zero and $\Omega_t$ is a Lipschitz domain, so
$v_t\in H^1_0(\Omega_t,\gamma)$. By
Lemma~\ref{lem:variational}, substituting $r=\rho s/R$,
\begin{align}
 \Tg(\Omega_t)
 &\geq \int_{\Omega_t}\bigl(2v_t-|\nabla v_t|^2\bigr)\,d\gamma \notag\\
 &=\frac1{2\pi}\int_0^{2\pi}\int_0^{\rho(\theta)}
 \left[
 2u\left(\frac{Rr}{\rho}\right)
 -\frac{R^2}{\rho^2}
  \left(1+\frac{\rho'^2}{\rho^2}\right)
  u'\left(\frac{Rr}{\rho}\right)^2
 \right]
 e^{-r^2/2}r\,dr\,d\theta \notag\\
 &=\frac1{2\pi}\int_0^{2\pi}
 \left[
 g\bigl(\rho(\theta)\bigr)
 -\frac{\rho'(\theta)^2}{\rho(\theta)^2}A_t(\theta)
 \right]d\theta,
 \label{eq:Jvt}
\end{align}
where
\begin{align*}
g(\rho)&=\int_0^R\Bigl[\frac{2\rho^2}{R^2}u(s)-u'(s)^2\Bigr]
e^{-\rho^2s^2/(2R^2)}s\,ds,
\\
0\le A_t(\theta)&=\int_0^R
 u'(s)^2
 e^{-\rho^2s^2/(2R^2)}
 s\,ds\le A=\int_0^Ru'(s)^2s\,ds .
\end{align*}
Note that $g(R)=\int_0^R(2u-u'^2)e^{-s^2/2}s\,ds=\Tg(B_R)$.

We claim that $g'(R)=c_R$. Differentiating under the integral,
\[
g'(R)=\frac1R\int_0^R\bigl[4u-s^2(2u-u'^2)\bigr]e^{-s^2/2}s\,ds .
\]
From \eqref{eq:radialODE} with $n=2$, $u''=-1-u'/s+su'$, hence
\[
\frac{d}{ds}\bigl[s^2u'(s)^2e^{-s^2/2}\bigr]
=\bigl[s^3u'^2-2s^2u'\bigr]e^{-s^2/2},
\ \ \text{so}\ \
R^2u'(R)^2e^{-R^2/2}=\int_0^R\bigl[s^3u'^2-2s^2u'\bigr]e^{-s^2/2}ds .
\]
Integrating by parts with $u(R)=0$,
$\int_0^R2s^2u'e^{-s^2/2}ds=-\int_0^R(4-2s^2)u\,s\,e^{-s^2/2}ds$.
Combining the two identities,
$\int_0^R[4u-2s^2u+s^2u'^2]e^{-s^2/2}s\,ds=R^2u'(R)^2e^{-R^2/2}$, that
is, $g'(R)=Ru'(R)^2e^{-R^2/2}=c_R$, where the last equality is
\eqref{eq:u2}.

Since $g$ is smooth near $R$, there exists $C_2$ with
$g(\rho)\ge g(R)+c_R(\rho-R)-C_2(\rho-R)^2$ for $|\rho-R|\le R/2$. We now estimate each of the terms in the expansion. First, we already noted that $g(R) = \Tg(B_R)$.  Secondly, by \eqref{eq:radial-i} and $h_t(\theta)=(1-t)R+t\,h_1(\theta)$, we have
\[
c_R(\rho_t-R) \ge c_R(h_t-R)-\frac{c_R\pi R_1^2}{R}t^2 =c_Rt (h_1 -R)- \frac{c_R\pi R_1^2}{R}t^2.
\]
Finally, 
$(1-t)B_R
\subset \Omega_t
\subset B_{(1-t)R+tR_1}$ gives that $C_2(\rho_t-R)^2
\leq
C_2(R+R_1)^2t^2.$
 Therefore,
\[
g(\rho_t)
\geq
g(R)+c_Rt(h_1-R)
-\left(
\frac{c_R\pi R_1^2}{R}
+C_2(R+R_1)^2
\right)t^2.
\]
By \eqref{eq:radial-ii} and $\rho_t\ge R/2$, the second term under the integral sign in
\eqref{eq:Jvt} is at most $(6R_1t/R)^2A$. 
Thus, the integrand in \eqref{eq:Jvt} can be bounded below by $g(R)+c_Rt(h_1-R)
-C t^2$ with 
$C=c_R\pi R_1^2/R+C_2(R+R_1)^2+36R_1^2A/R^2$ provided $t_1$ is small enough so
that $|\rho_t-R|\le R/2$ for $0<t\le t_1$. Inserting this last estimate into \eqref{eq:Jvt} completes the proof.
\end{proof}

\begin{remark}\label{rem:hadamard}
The identity $g'(R)=c_R$ is the formal Hadamard formula for the domain
variation of $\Tg$: since $v_t$ agrees with the torsion function at $t=0$
and $\int_{B_R}\bigl(2v-|\nabla v|^2\bigr)\,d\gamma$ is stationary there,
the expected first-order term is
$\int_{\partial B_R}u_\nu^2\,(\delta h)\,d\gamma_{\partial B_R}$, with
$\delta h=h_1-R$ and $u_\nu=u_2'(R)$ constant on $\partial B_R$.
Proposition~\ref{prop:variation} proves the corresponding one-sided
expansion, a lower bound whose first-order term depends on the perturbing
body only through its perimeter. This is all the counterexample requires. We do not prove that $c_RM$ is the exact first variation, which would
follow from a Hadamard shape-derivative theorem for sufficiently regular
variations.
\end{remark}

Remark~\ref{rem:hadamard} explains the failure of \eqref{eq:BMconj}. The
perimeter is linear under Minkowski addition,
$\Per((1-t)\Omega_0+t\Omega_1)=(1-t)\Per(\Omega_0)+t\Per(\Omega_1)$, so
along the segment from a ball to a body of larger perimeter the perimeter
increases linearly and the torsion increases at least at the positive
rate $c_R/(2\pi)$ per unit of added perimeter. A body of large perimeter and small torsion,
such as the needle in Theorem~\ref{thm:main-BM}, 
therefore increases 
the torsion $\Tg$  of the Minkowski sum at first order while contributing essentially nothing through its own torsion to the conjectured right-hand side
\eqref{eq:BMconj}. Along the family of balls, perimeter and torsion
increase together, and \eqref{eq:BMconj} holds \cite{MSS26}.

\subsection{Proof of Theorem \ref{thm:main-BM} for $n=2$}\label{sec:needle}

\begin{lemma}\label{lem:needle}
Let $\Omega\subset\{|x_2|<\varepsilon\}\subset\R^n$, $n\geq 2$, be open and bounded,
with $0<\varepsilon\le\frac12$. Then
$\Tg(\Omega)\le\dfrac{\varepsilon^2}{2(1-\varepsilon^2)}$.
\end{lemma}

\begin{proof}
Let $q(x)=\dfrac{\varepsilon^2-x_2^2}{2(1-\varepsilon^2)}$. Then
$-\Lou q=\dfrac{1-x_2^2}{1-\varepsilon^2}\ge1$ on $\{|x_2|<\varepsilon\}$
and $q\ge0$ on $\overline\Omega$, so $u_\Omega\le q$ by the maximum
principle and $\Tg(\Omega)\le\sup q\cdot\gamma(\Omega)\le\sup q$.
\end{proof}

\begin{proof}[Proof of Theorem \ref{thm:main-BM}, case $n=2$]
Let $R=1$, $\Omega_0=B_1$, and $\Omega_1=N=[-L,L]\times\{0\}+
B_\varepsilon$ with $L=2$, $\varepsilon=\frac1{10}$. The support function
of $N$ is $h_N(\theta)=L|\cos\theta|+\varepsilon$, so
$\Per(N)=4L+2\pi\varepsilon=8+\frac\pi5>2\pi=\Per(B_1)$, because
$\pi<\frac{16}5$. By Proposition~\ref{prop:variation}, there exists $t_1>0$
with $\Tg(\Omega_t)>\Tg(B_1)$ for $0<t<t_1$, where
$\Omega_t=(1-t)\Omega_0+t\Omega_1$. Since $N\subset\{|x_2|<\frac1{10}\}$,
Lemma~\ref{lem:needle} gives $\Tg(N)\le\frac1{198}$, while
Lemma~\ref{lem:ballclosed} with $Z=\frac12$ gives
$\Tg(B_1)\ge Z^2/4=\frac1{16}$, the series having positive terms. Hence
$\Tg(\Omega_t)>\Tg(B_1)=\max(\Tg(\Omega_0),\Tg(\Omega_1))$ for
$0<t<t_1$, and for every $\alpha>0$,
\[
\Tg(\Omega_t)^\alpha>\max\bigl(\Tg(\Omega_0),\Tg(\Omega_1)\bigr)^\alpha
\ge(1-t)\Tg(\Omega_0)^\alpha+t\Tg(\Omega_1)^\alpha .
\]

For the smooth version, let $N_\delta$ be centrally symmetric convex bodies
of class $C^\infty_+$ with $d_H(N_\delta,N)\le\delta$
(Section~\ref{sec:convex}). Then $h_{N_\delta}\to h_N$ uniformly, so
$\Per(N_\delta)\to\Per(N)>2\pi$, and
$N_\delta\subset N+B_\delta\subset\{|x_2|<\varepsilon+\delta\}$, so
$\Tg(N_\delta)<\frac1{16}$ for small $\delta$, by Lemma~\ref{lem:needle}.
The argument above applies to the pair $(B_1,N_\delta)$.
\end{proof}

\section{Higher dimensions}\label{sec:higherdim}

In $\R^n$ the role of the perimeter is played by the mean width. Let
$\kappa_n$ be the volume of the unit ball, so that the surface measure
$\sigma$ of $\mathbb{S}^{n-1}$ has total mass $n\kappa_n$, and let
\[
w(K)=\frac{2}{n\kappa_n}\int_{\mathbb{S}^{n-1}}h_K\,d\sigma
\]
be the mean width of a convex body $K$, so $w(B_R)=2R$. Like the
perimeter in the plane, $w$ is linear under Minkowski addition, and
Cauchy's formula gives $w(K)=\Per(K)/\pi$ for $n=2$.

\begin{theorem}\label{thm:main-n}
Let $n\ge3$, $\Omega_0=B_1\subset\R^n$, and
$\Omega_1=([-L,L]\times\{0\}^{n-1})+B_\varepsilon$, where
\[
L=\frac{n\kappa_n}{2\kappa_{n-1}},\qquad
0<\varepsilon\le\frac12,\qquad
\varepsilon^2<\frac{3}{4n}\int_{B_1}\bigl(1-|x|^2\bigr)\,d\gamma .
\]
Then there exists $t_1>0$ such that
\[
\Tg\bigl((1-t)\Omega_0+t\Omega_1\bigr)>
\max\bigl(\Tg(\Omega_0),\Tg(\Omega_1)\bigr)
\qquad\text{for all }t\in(0,t_1).
\]
Consequently, for every $\alpha>0$, the function
$t\mapsto\Tg((1-t)\Omega_0+t\Omega_1)^{\alpha}$ is not convex on $[0,1]$.
In particular, the inequality of Question~(Q) of \cite{MSS26} with
exponent $1/(n+2)$ fails. The same conclusion holds with $\Omega_1$ replaced by a
centrally symmetric convex body of class $C^\infty_+$. For $n=3$, one may
take $L=2$ and $\varepsilon=\frac1{10}$.
\end{theorem}

The proof follows Section~\ref{sec:BM} for which we list the corresponding modifications. Let
$u_n$ solve \eqref{eq:radialODE}, so that the torsion function of
$B_R\subset\R^n$ is $u=u_n(|x|)-u_n(R)$, and
$u_n'(r)=-r^{1-n}e^{r^2/2}\int_0^r s^{n-1}e^{-s^2/2}\,ds<0$, for $r>0$.

\begin{lemma}\label{lem:radial-n}
Lemma~\ref{lem:radial} holds in $\R^n$: with $\Omega_1\subset
B_{R_1}\subset\R^n$ a convex body containing the origin,
$0<t\le t_0=\min\{\frac12,\frac{R}{8R_1}\}$, and $\rho_t$ the radial
function of $\Omega_t=(1-t)B_R+t\Omega_1$ on $\mathbb{S}^{n-1}$, the estimate
\eqref{eq:radial-i} holds for all $\xi\in \mathbb{S}^{n-1}$ and
$|\nabla_S\rho_t|\le3R_1t$ almost everywhere, where $\nabla_S$ is the
gradient on $\mathbb{S}^{n-1}$.
\end{lemma}

\begin{proof}
The proof of Lemma~\ref{lem:radial} applies with three changes. Let
$x=\rho_t(\xi)\xi\in\partial\Omega_t$ with outer unit normal $\nu$ and
decomposition $x=(1-t)R\nu+t\zeta$, $\zeta\in\partial\Omega_1$. First,
the cross product is replaced by the tangential component: since
$x-(x\cdot\nu)\nu=t\bigl(\zeta-(\zeta\cdot\nu)\nu\bigr)$ has norm at most
$tR_1$, the angle $\beta$ between $\xi$ and $\nu$ satisfies
$\sin\beta=|x-(x\cdot\nu)\nu|/|x|\le tR_1/((1-t)R)$, and \eqref{eq:beta}
holds. Second, the bound $h_t(\nu)\ge h_t(\xi)-tR_1\beta$ follows since
$h_1$ is $R_1$-Lipschitz, being the support function of a subset of
$B_{R_1}$, and $|\nu-\xi|=2\sin(\beta/2)\le\beta$. Third, at points of
differentiability of $\rho_t$, the outer normal is
$\nu=(\rho_t\xi-\nabla_S\rho_t)/(\rho_t^2+|\nabla_S\rho_t|^2)^{1/2}$, so
$\cos\beta=\nu\cdot\xi$ gives $|\nabla_S\rho_t|=\rho_t\tan\beta$, and the
constants are as before.
\end{proof}

\begin{proposition}\label{prop:variation-n}
Let $\Omega_1\subset B_{R_1}\subset\R^n$ be a convex body containing the
origin, $R>0$, and
\[
M=\frac{1}{n\kappa_n}\int_{\mathbb{S}^{n-1}}\bigl(h_1-R\bigr)\,d\sigma
=\frac{w(\Omega_1)-w(B_R)}{2},
\qquad
c_{R,n}=\frac{n\kappa_n}{(2\pi)^{n/2}}\,R^{n-1}u_n'(R)^2e^{-R^2/2}.
\]
Then there exist $t_1,C>0$ depending only on $n,R,R_1$ such that
\[
\Tg\bigl((1-t)B_R+t\Omega_1\bigr)\ge\Tg(B_R)+c_{R,n}Mt-Ct^2
\qquad\text{for }0<t\le t_1 .
\]
In particular, if $w(\Omega_1)>2R$, then
$\Tg((1-t)B_R+t\Omega_1)>\Tg(B_R)$ for all sufficiently small $t>0$.
\end{proposition}

For $n=2$, $c_{R,2}=c_R$ and $M$ agrees with
Proposition~\ref{prop:variation} by Cauchy's formula.

\begin{proof}
It suffices to prove the claim for some $t_1 \leq t_0$, with $t_0$ as in
Lemma~\ref{lem:radial-n}. With $u=u_n-u_n(R)$ and the trial function
$v_t(r\xi)=u(Rr/\rho_t(\xi))$, $\xi\in \mathbb{S}^{n-1}$, the computation of
$|\nabla v_t|^2$ is unchanged, and the substitution $r=\rho s/R$ now
carries the factor $r^{n-1}\,dr=(\rho/R)^ns^{n-1}\,ds$:
\[
\Tg(\Omega_t)\ge\int_{\Omega_t}\bigl(2v_t-|\nabla v_t|^2\bigr)\,d\gamma
=\frac{1}{(2\pi)^{n/2}}\int_{\mathbb{S}^{n-1}}
\Bigl[G\bigl(\rho_t(\xi)\bigr)
-\frac{|\nabla_S\rho_t|^2}{\rho_t^2}
\Bigl(\frac{\rho_t}{R}\Bigr)^{n-2}A_t(\xi)\Bigr]\,d\sigma(\xi),
\]
where
\[
G(\rho)=\int_0^R\Bigl[\frac{2\rho^n}{R^n}u(s)
-\frac{\rho^{n-2}}{R^{n-2}}u'(s)^2\Bigr]
e^{-\rho^2s^2/(2R^2)}s^{n-1}\,ds,
\]
\[
A_t(\xi)=\int_0^R u'(s)^2
e^{-\rho_t(\xi)^2s^2/(2R^2)}s^{n-1}\,ds,
\qquad
0\le A_t(\xi)\le A:=\int_0^Ru'(s)^2s^{n-1}\,ds ,
\]
and $(\rho_t/R)^{n-2}\le(3/2)^{n-2}$ for $|\rho_t-R|\le R/2$. Then
$(2\pi)^{-n/2}n\kappa_n\,G(R)=\Tg(B_R)$, and we claim that
$G'(R)=R^{n-1}u'(R)^2e^{-R^2/2}$, so that
$(2\pi)^{-n/2}n\kappa_n\,G'(R)=c_{R,n}$. Differentiating under the
integral,
\[
G'(R)=\frac1R\int_0^R\bigl[2nu-(n-2)u'^2-s^2(2u-u'^2)\bigr]
e^{-s^2/2}s^{n-1}\,ds .
\]
From \eqref{eq:radialODE}, $u''=-1-(n-1)u'/s+su'$, hence
\[
\frac{d}{ds}\bigl[s^nu'(s)^2e^{-s^2/2}\bigr]
=\bigl[(2-n)s^{n-1}u'^2+s^{n+1}u'^2-2s^nu'\bigr]e^{-s^2/2},
\]
and integrating from $0$ to $R$ (the boundary term at $0$ vanishes since
$u'(s)\sim-s/n$),
\[
R^nu'(R)^2e^{-R^2/2}
=\int_0^R\bigl[(2-n)s^{n-1}u'^2+s^{n+1}u'^2-2s^nu'\bigr]e^{-s^2/2}\,ds .
\]
Integrating by parts with $u(R)=0$,
$\int_0^R2s^nu'e^{-s^2/2}ds=-\int_0^R(2n-2s^2)u\,s^{n-1}e^{-s^2/2}ds$.
Combining the two identities gives
$R^nu'(R)^2e^{-R^2/2}=RG'(R)$, which is the claim. The estimates of
Lemma~\ref{lem:radial-n} now give, as in the proof of
Proposition~\ref{prop:variation},
\[
G\bigl(\rho_t(\xi)\bigr)\ge G(R)+G'(R)\,t\,\bigl(h_1(\xi)-R\bigr)-C_2t^2,
\]
where $C_2$ now absorbs both the Taylor remainder of $G$ and the $t^2$
error in \eqref{eq:radial-i}, and
with $A$ replaced by $(3/2)^{n-2}A$ in the bound for the angular term,
and averaging over $\mathbb{S}^{n-1}$ yields
\begin{align*}
\int_{\Omega_t}\bigl(2v_t-|\nabla v_t|^2\bigr)\,d\gamma
&\ge\frac{n\kappa_n}{(2\pi)^{n/2}}\,G(R)
+\frac{G'(R)\,t}{(2\pi)^{n/2}}\int_{\mathbb{S}^{n-1}}\bigl(h_1-R\bigr)\,d\sigma
-Ct^2\\
&=\Tg(B_R)+c_{R,n}Mt-Ct^2.\qedhere
\end{align*}
\end{proof}

The explicit series of Lemma~\ref{lem:ballclosed} is replaced by the
following lower bound.

\begin{lemma}\label{lem:ball-lower}
For $B_R\subset\R^n$,
$\Tg(B_R)\ge\dfrac1{2n}\displaystyle\int_{B_R}\bigl(R^2-|x|^2\bigr)\,d\gamma$.
\end{lemma}

\begin{proof}
The classical torsion function $q(x)=(R^2-|x|^2)/(2n)$ satisfies
$-\Lou q=1-|x|^2/n\le1$ in $B_R$ and $q=0$ on $\partial B_R$, so
$u_{B_R}\ge q$ by the maximum principle. We finish the proof by integrating.
\end{proof}

\begin{proof}[Proof of Theorem \ref{thm:main-n}]
The support function of
$\Omega_1=N=([-L,L]\times\{0\}^{n-1})+B_\varepsilon$ is
$h_N(\xi)=L|\xi_1|+\varepsilon$. The divergence theorem applied to
the constant field $(1,0,\dots,0)$ on the half-ball
$\{x\in B_1:x_1\ge0\}$
gives $\int_{\mathbb{S}^{n-1}}(\xi_1)^+\,d\sigma=\kappa_{n-1}$, hence
$\int_{\mathbb{S}^{n-1}}|\xi_1|\,d\sigma=2\kappa_{n-1}$ and, with
$L=n\kappa_n/(2\kappa_{n-1})$ and $R=1$,
\[
M=\frac{1}{n\kappa_n}\int_{\mathbb{S}^{n-1}}\bigl(h_N-1\bigr)\,d\sigma
=\frac{2L\kappa_{n-1}}{n\kappa_n}+\varepsilon-1=\varepsilon>0 .
\]
By Proposition~\ref{prop:variation-n}, there is $t_1>0$ with
$\Tg(\Omega_t)>\Tg(B_1)$ for $0<t<t_1$, where
$\Omega_t=(1-t)\Omega_0+t\Omega_1$. Since $N\subset\{|x_2|<\varepsilon\}$
and $\varepsilon\le\frac12$, Lemma~\ref{lem:needle} and the hypothesis on
$\varepsilon$ give
\[
\Tg(N)\le\frac{\varepsilon^2}{2(1-\varepsilon^2)}
\le\frac{2\varepsilon^2}{3}
<\frac1{2n}\int_{B_1}\bigl(1-|x|^2\bigr)\,d\gamma\le\Tg(B_1),
\]
using Lemma~\ref{lem:ball-lower} in the last step. Hence
$\Tg(\Omega_t)>\Tg(B_1)=\max(\Tg(\Omega_0),\Tg(\Omega_1))$ for
$0<t<t_1$, and the non-convexity of $t\mapsto\Tg(\Omega_t)^\alpha$
follows as in the proof of Theorem~\ref{thm:main-BM}. The smooth version
follows by the approximation argument there: $w(N_\delta)\to w(N)$, and
since the hypothesis on $\varepsilon$ is strict, $\delta>0$ can be chosen
so small that $\varepsilon+\delta\le\frac12$,
$(\varepsilon+\delta)^2<\frac{3}{4n}\int_{B_1}(1-|x|^2)\,d\gamma$, and
$N_\delta\subset\{|x_2|<\varepsilon+\delta\}$.

For $n=3$, $L=n\kappa_n/(2\kappa_{n-1})=2$, and $\varepsilon=\frac1{10}$
is admissible: since $1-|x|^2\ge0$ and $e^{-|x|^2/2}\ge e^{-1/2}$ on
$B_1$,
\[
\frac{3}{4n}\int_{B_1}\bigl(1-|x|^2\bigr)\,d\gamma
\ge\frac14\sqrt{\frac{2}{\pi}}\,e^{-1/2}\int_0^1(1-s^2)s^2\,ds
=\frac{1}{30}\sqrt{\frac{2}{\pi}}\,e^{-1/2}>\frac1{100},
\]
because $9\pi e<200$.
\end{proof}

The first variation at the ball suggests that the ball converts perimeter
into torsion most efficiently. In the classical setting, the corresponding
statements are Saint-Venant's inequality for the torsional rigidity at fixed area \cite{Pol48}, \cite{Mak66} and,
for the first Steklov eigenvalue, Weinstock's inequality
$\sigma_1\Per\le2\pi$ for simply connected plane domains, with equality
only for the disk \cite{Wei54}, \cite[\S7.3]{Hen06}. For $\Tg$, no
symmetrization argument is available, since Gaussian symmetrization moves
mass towards half-spaces rather than balls and the operator $\Lou$ is not
scale invariant. We do not know whether
\[
\Tg(\Omega)\le\Tg(B_R)\qquad\text{whenever }\Per(\Omega)=\Per(B_R)
\]
holds for centrally symmetric convex plane bodies $\Omega$, nor whether
its mean-width analogue, $\Tg(\Omega)\le\Tg(B_R)$ whenever $w(\Omega)=2R$,
holds for centrally symmetric convex bodies $\Omega\subset\R^n$.
Propositions~\ref{prop:variation} and~\ref{prop:variation-n} show that
$\Tg$ increases to first order at the ball whenever the perimeter,
respectively the mean width, does.

\section{Dimension one}\label{sec:oneD}

For $n=1$, Question~(Q) of \cite{MSS26} asks whether
$t\mapsto\Tg((1-t)\Omega_0+t\Omega_1)^{1/3}$ is convex for open, bounded,
centrally symmetric $\Omega_0,\Omega_1\subset\R$. A connected such set is
a symmetric interval $B_R=(-R,R)$. For such sets, the question is already settled in \cite[Theorem 1.2, Corollary 1.3]{MSS26}.
Theorem~\ref{thm:oneD} (i), (ii) below records the one-dimensional case from \cite{MSS26} and (iii) adds
that the logarithm is neither convex nor concave. Furthermore in (iv), no non-zero
exponent yields concavity. What Question~(Q) leaves open for $n=1$ is the class of centrally
symmetric sets that are not connected, and there convexity fails for every positive exponent, see Proposition~\ref{prop:disconnected}. Proposition~\ref{prop:noncentered} shows, in addition, that the central-symmetry hypothesis
in Question~(Q) cannot be dropped: convexity fails for every positive exponent already for a pair of reflected, off-center intervals.

Throughout this section, set
\[
\Phi(R)=\int_0^R e^{-s^2/2}\,ds,\qquad
T(R)=\Tg(B_R).
\]

\begin{lemma}\label{lem:closed-oneD}
For $R>0$,
\[
T(R)=\sqrt{\tfrac2\pi}\,J(R),\qquad
J(R)=\int_0^R e^{y^2/2}\Phi(y)^2\,dy .
\]
\end{lemma}

\begin{proof}
For $n=1$, \eqref{eq:radialODE} reads
$(e^{-r^2/2}u_1')'=-e^{-r^2/2}$, so $u_1'(r)=-e^{r^2/2}\Phi(r)$, and the
torsion function of $B_R$ is
\[
u(x)=u_1(|x|)-u_1(R)=\int_{|x|}^R e^{y^2/2}\Phi(y)\,dy .
\]
Since $u$ is even, Fubini's theorem gives
\[
T(R)=2\int_0^R u(x)\,\frac{e^{-x^2/2}}{\sqrt{2\pi}}\,dx
=\sqrt{\tfrac2\pi}\int_0^R e^{y^2/2}\Phi(y)
\Bigl(\int_0^y e^{-x^2/2}\,dx\Bigr)dy
=\sqrt{\tfrac2\pi}\,J(R). \qedhere
\]
\end{proof}

\begin{theorem}\label{thm:oneD}
For $R_0,R_1>0$ and $t\in[0,1]$ let $R_t=(1-t)R_0+tR_1$, so that
$(1-t)B_{R_0}+tB_{R_1}=B_{R_t}\subset\R$.
\begin{enumerate}
\item[(i)] For every $\alpha\ge\frac13$, the function
$t\mapsto T(R_t)^{\alpha}$ is convex on $[0,1]$ for all $R_0,R_1>0$, and
strictly convex when $R_0\neq R_1$.
\item[(ii)] For every $\alpha\in(0,\frac13)$ there exist $R_0\neq R_1$
for which $t\mapsto T(R_t)^{\alpha}$ is not convex.
\item[(iii)] The function $t\mapsto\log T(R_t)$ is neither convex for
all pairs $R_0,R_1$ nor concave for all pairs.
\item[(iv)] For every $\alpha\in\R\setminus\{0\}$ there exist
$R_0\neq R_1$ for which $t\mapsto T(R_t)^{\alpha}$ is not concave.
\end{enumerate}
\end{theorem}

\begin{proof}
Parts (i) and (ii) are proved in \cite{MSS26}: Theorem~1.2 gives the
case $\alpha=\frac13$ with equality only for $R_0=R_1$, the case
$\alpha>\frac13$ follows by composing with the increasing, strictly
convex map $x\mapsto x^{3\alpha}$, as observed there, and Corollary~1.3
shows that the exponent cannot be lowered for $n=1$. 

We prove (iii). Since $\Phi^2<\frac\pi2$ and
$\int_0^S e^{y^2/2}\,dy\le e^{S^2/2}$ for all $S>0$
(for \(0<S\le1\),
\(
\int_0^S e^{y^2/2}\,dy
\le Se^{S^2/2}\le e^{S^2/2},
\)
and for \(S\ge1\),
\(
\int_0^S e^{y^2/2}\,dy
=\int_0^1 e^{y^2/2}\,dy+\int_1^S e^{y^2/2}\,dy
\le e^{1/2}+\int_1^S y e^{y^2/2}\,dy
=e^{S^2/2}\)),
 while
$\Phi(1)\ge e^{-1/2}$, we have
\begin{equation}\label{eq:Jbounds}
J(S)\le\frac\pi2\,e^{S^2/2}\quad(S>0),
\qquad
J(S)\ge\int_{S-1}^S e^{y^2/2}\Phi(y)^2\,dy\ge e^{-1}e^{(S-1)^2/2}
\quad(S\ge2).
\end{equation}
Let $R_0=1$, $R_1=L\ge3$ and $t=\frac12$, so $R_{1/2}=\frac{1+L}2\ge2$.
Concavity of $t\mapsto\log T(R_t)$ would give
$J(R_{1/2})^2\ge J(1)J(L)$, while by \eqref{eq:Jbounds}
\[
\frac{J(R_{1/2})^2}{J(1)J(L)}
\le\frac{(\pi/2)^2e}{J(1)}\,e^{(1+L)^2/4-(L-1)^2/2}
=\frac{(\pi/2)^2e}{J(1)}\,e^{(-L^2+6L-1)/4}
\xrightarrow[L\to\infty]{}0 ,
\]
a contradiction for large $L$.
Finally, $J'(R)=e^{R^2/2}\Phi(R)^2$ by definition, so $J''=RJ'+2\Phi$,
and $\Phi(y)=y+O(y^3)$ gives $J=\frac{R^3}3+O(R^5)$, $J'=R^2+O(R^4)$ and
$J''=2R+O(R^3)$, hence
\begin{equation}\label{eq:psi}
\psi:=\frac{TT''}{(T')^2}=\frac{JJ''}{(J')^2}\to\frac23
\qquad\text{as }R\to0^+ .
\end{equation}
Since $(\log T)''=(T'/T)^2(\psi-1)$, $\log T$ is strictly concave near
$0$, hence $t\mapsto\log T(R_t)$ is not convex for pairs of small
intervals.

Part (iv) follows from (iii). If $t\mapsto T(R_t)^{\alpha}$ is concave
for a pair $R_0\neq R_1$, then so is
$\log T(R_t)^{\alpha}=\alpha\log T(R_t)$, because the logarithm is
increasing and concave. For $\alpha>0$, concavity for the pair $R_0=1$,
$R_1=L$ above would thus make $t\mapsto\log T(R_t)$ concave for that
pair, which was excluded for large $L$, while for $\alpha<0$, concavity for a
pair of small intervals would make $t\mapsto\log T(R_t)$ convex,
contradicting the strict concavity of $\log T$ near $0$.
\end{proof}

Connectedness cannot be removed from Theorem~\ref{thm:oneD}: when one
of the sets is a union of two intervals, convexity fails for every
positive exponent. Its mechanism is that Minkowski addition merges the
two components: for $t<\frac\delta2$ the combination below is a single
interval strictly containing $\Omega_0$, while $\Tg(\Omega_1)$ is small
because the components of $\Omega_1$ lie far from the origin.

\begin{proposition}\label{prop:disconnected}
For $0<\delta\le\frac1{10}$, let $\Omega_0=(-\delta,\delta)$ and
\[
\Omega_1
=\bigl(-2-\delta,-2+\delta\bigr)
\cup\bigl(2-\delta,2+\delta\bigr).
\]
Then, with $\Omega_t=(1-t)\Omega_0+t\Omega_1$,
\[
\Tg(\Omega_t)>\max\bigl(\Tg(\Omega_0),\Tg(\Omega_1)\bigr)
\qquad\text{for }0<t<\tfrac\delta2 .
\]
In particular, for every $\alpha>0$ the function
$t\mapsto\Tg(\Omega_t)^{\alpha}$ is not convex on $[0,1]$, and neither
is $t\mapsto\log\Tg(\Omega_t)$, and the convexity of Question~(Q) fails,
for every positive exponent, on the class of open, bounded, centrally
symmetric subsets of $\R$.
\end{proposition}

\begin{proof}
We have $\Omega_t=(-2t-\delta, -2t + \delta) \cup (2t-\delta, 2t+\delta)$ because Minkowski addition distributes over unions. For $0<t<\frac\delta2$ they overlap, and
\[
\Omega_t=\bigl(-2t-\delta, 2t+ \delta\bigr)
=B_{\delta+2t} ,
\]
so Lemma~\ref{lem:closed-oneD} gives
\[
\Tg(\Omega_t)-\Tg(\Omega_0)
=\sqrt{\tfrac2\pi}\int_{\delta}^{\delta+2t}e^{y^2/2}\Phi(y)^2\,dy>0 .
\]

It remains to show $\Tg(\Omega_1)<\Tg(\Omega_0)$. Since
$\Phi(y)\ge ye^{-y^2/2}$, we have
$e^{y^2/2}\Phi(y)^2\ge y^2e^{-y^2/2}\ge y^2e^{-1/200}$ on $[0,\delta]$,
so Lemma~\ref{lem:closed-oneD} gives
\[
\Tg(\Omega_0)=\sqrt{\tfrac2\pi}\,J(\delta)
\ge\sqrt{\tfrac2\pi}\,\frac{\delta^3}3\,e^{-1/200} .
\]
The components of $\Omega_1$ are disjoint, so the torsion function of
$\Omega_1$ restricts on each component to the torsion function of that
component, and $\Tg(\Omega_1)=2\,\Tg(I)$ with $I=(2-\delta,2+\delta)$,
by the symmetry of $\gamma$. On $I$, the barrier
$q(x)=\frac{50}{79}\bigl(\delta^2-(x-2)^2\bigr)$ satisfies $q\ge0$ on
$\overline I$ with $q=0$ at the endpoints, and
$x(x-2)\le2\delta+\delta^2\le\frac{21}{100}$, so
\[
-\Lou q=\frac{100}{79}\bigl(1-x(x-2)\bigr)
\ge\frac{100}{79}\cdot\frac{79}{100}=1 ,
\]
hence $u_{\Omega_1}\le q\le\frac{50}{79}\delta^2$ on $I$ by the maximum
principle. Since the density of $\gamma$ is at most
$(2\pi)^{-1/2}e^{-(2-\delta)^2/2}\le(2\pi)^{-1/2}e^{-361/200}$ on $I$,
\[
\Tg(\Omega_1)=2\,\Tg(I)
\le2\cdot\frac{50}{79}\,\delta^2\,\gamma(I)
\le\sqrt{\tfrac2\pi}\,\frac{100}{79}\,\delta^3 e^{-361/200}
<\sqrt{\tfrac2\pi}\,\frac{\delta^3}3\,e^{-1/200}\le\Tg(\Omega_0),
\]
where the strict inequality is $e^{9/5}>\frac{300}{79}$, which holds
since $e^3>16>\bigl(\tfrac{300}{79}\bigr)^2$. Hence
$\Tg(\Omega_t)>\Tg(\Omega_0)>\Tg(\Omega_1)$ for
$0<t<\frac\delta2$, and for every $\alpha>0$,
\[
\Tg(\Omega_t)^{\alpha}
>\max\bigl(\Tg(\Omega_0),\Tg(\Omega_1)\bigr)^{\alpha}
\ge(1-t)\,\Tg(\Omega_0)^{\alpha}+t\,\Tg(\Omega_1)^{\alpha},
\]
and likewise for the logarithm.
\end{proof}

Central symmetry of the individual sets cannot be dropped either, even
for intervals. Example~3.1 of \cite{MSS26} observes this numerically for
$\Omega_0=(-1,1)$, $\Omega_1=(-2,0)$ and $\alpha=1$. The following
proposition gives a rigorous instance that rules out every positive exponent. The idea is that sliding an interval away from the
origin decreases its torsion, so the Minkowski midpoint of an off-center
interval and its reflection has larger torsion than both endpoints.

\begin{proposition}\label{prop:noncentered}
For $0<\delta\le\frac1{10}$, let $\Omega_0=(2-\delta,2+\delta)$ and
$\Omega_1=-\Omega_0$, so that
$\tfrac12\Omega_0+\tfrac12\Omega_1=B_{\delta}$. Then
\[
\Tg(B_{\delta})>\Tg(\Omega_0)=\Tg(\Omega_1),
\]
so for every $\alpha>0$ the function
$t\mapsto\Tg\bigl((1-t)\Omega_0+t\Omega_1\bigr)^{\alpha}$ is not convex
on $[0,1]$, and neither is $t\mapsto\log\Tg\bigl((1-t)\Omega_0+t\Omega_1\bigr)$.
\end{proposition}

\begin{proof}
Since $\gamma$ is invariant under $x\mapsto-x$,
$\Tg(\Omega_1)=\Tg(\Omega_0)$, and the estimates in the proof of
Proposition~\ref{prop:disconnected} give
\[
\Tg(\Omega_0)=\Tg(I)
\le\sqrt{\tfrac2\pi}\,\frac{50}{79}\,\delta^3 e^{-361/200}
<\sqrt{\tfrac2\pi}\,\frac{\delta^3}3\,e^{-1/200}\le\Tg(B_\delta) .
\]
Finally, at $t=\frac12$, for every $\alpha>0$,
\[
\Tg(B_{\delta})^{\alpha}>\Tg(\Omega_0)^{\alpha}
=\tfrac12\,\Tg(\Omega_0)^{\alpha}+\tfrac12\,\Tg(\Omega_1)^{\alpha},
\]
and likewise for the logarithm.
\end{proof}

\section{Acknowledgment}
In our first version of the manuscript, we misquoted Mar\'in Sola and Salerno. We thank them for pointing it out. Their conjecture 1.4 is stated for all dimensions $n$ rather than just $n=2$ and the Question (Q) is for $n \in \{1,2\}$. After their communication, we added the section about dimension one.
\section{Declaration of AI usage}
The counterexamples in this paper were found after a computer-assisted
search with Claude Fable 5.0 that included symbolic calculations, numerical counterexample
searches, proof exploration, and proof review. Further verification was performed with ChatGPT Sol 5.6. AI output was not treated as mathematical authority. The authors checked, edited, and remain responsible for all theorem statements, proofs, citations and final prose.

\end{document}